\documentclass{amsart}

\usepackage{amssymb}
\usepackage{amsfonts}
\usepackage[english,german]{babel}
\usepackage{booktabs}
\usepackage{color}
\usepackage{graphicx}
\usepackage[utf8]{inputenc}
\usepackage{mathrsfs}
\usepackage{subfigure}
\usepackage{tensor}

\usepackage[
		pdftex,
		bookmarks,
		bookmarksopen,
		bookmarksopenlevel=1,
		bookmarksnumbered,
		breaklinks,
		colorlinks,
		linkcolor=linkcolor,
		citecolor=citecolor,
		urlcolor=urlcolor,
	]{hyperref}

\newcommand{\comma}{,}
\hypersetup{
	pdftitle    = Separation coordinates\comma\ moduli spaces and Stasheff polytopes,%
	pdfauthor   = Konrad Sch\"obel \& Alexander P. Veselov,%
	pdfsubject  = Riemannian geometry\comma\ algebraic geometry,%
	pdfcreator  = \LaTeX{} with `amsart' class,%
	pdfproducer = pdf\LaTeX,%
	pdfkeywords =
		separation of variables\comma\
		St\"ackel systems\comma\
		Deligne-Mumford moduli spaces\comma\
		stable curves with marked points\comma\
		associahedra\comma\
		Stasheff polytopes\comma\
		mosaic operad%
}
\definecolor{linkcolor}{rgb}{0.6,0.0,0.0}
\definecolor{citecolor}{rgb}{0.0,0.6,0.0}
\definecolor{urlcolor} {rgb}{0.0,0.0,0.6}

\newcommand{\beq}{\begin{equation}}
\newcommand{\eeq}{\end{equation}}

\newcommand{\dfn}[1]{\emph{#1}}
\renewcommand{\ge}{\geqslant}
\renewcommand{\le}{\leqslant}
\newcommand{\OG}{\operatorname{O}}

\newtheorem{prop}{Proposition}
\newtheorem*{mthm}{Main Theorem}
\newtheorem{thm}{Theorem}
\newtheorem{cor}{Corollary}
\theoremstyle{definition}
\newtheorem{definition}{Definition}
\theoremstyle{remark}
\newtheorem{rem}{Remark}
\numberwithin{equation}{section}

\begin{document}

\selectlanguage{english}

\title{Separation coordinates, moduli spaces and Stasheff polytopes}

\author{K. Sch\"obel}
\address{Friedrich-Schiller-Universit\"at Jena,
	Fakult\"at f\"ur Mathematik und Informatik,
	Ernst-Abbe-Platz 2,
	07743 Jena, Germany}
\email{konrad.schoebel@uni-jena.de}

\author{A. P. Veselov}
\address{Department of Mathematics,
	Loughborough University, Loughborough,
	Leicestershire, LE11 3TU, UK and
	Moscow State University, Russia}
\email{A.P.Veselov@lboro.ac.uk}

\subjclass[2010]{
	Primary
	14D21;  
	Secondary
	53A60,  
	58D27
}

\keywords{
	Separation of variables,
	St\"ackel systems,
	Deligne-Mumford moduli spaces,
	stable curves with marked points,
	associahedra,
	Stasheff polytopes,
	mosaic operad%
}

\begin{abstract}
	We show that the orthogonal separation coordinates on the sphere $S^n$ are naturally
	parametrised by the real version of the Deligne-Mumford-Knudsen moduli
	space $\bar M_{0,n+2}(\mathbb R)$ of stable curves of genus zero with
	$n+2$ marked points. We use the combinatorics of Stasheff polytopes
	tessellating $\bar M_{0,n+2}(\mathbb R)$ to classify the different
	canonical forms of separation coordinates and deduce an explicit
	construction of separation coordinates as well as of St\"ackel systems from the
	mosaic operad structure on $\bar M_{0,n+2}(\mathbb R)$.
\end{abstract}

\maketitle

\section{Introduction}

Separation of variables is one of the oldest techniques in mathematical physics,
which still remains one of the most effective and powerful tools in the theory of integrable systems.
Its quantum version initiated by Lam\'e \cite{Lame} is actually more natural than the classical one developed
approximately at the same time by Jacobi \cite{Jac},
when one has to consider the Hamilton-Jacobi equation rather than the equations of motion.

The general theory of separation coordinates goes back to St\"ackel \cite{Staeckel} and Levi-Civita \cite{LC}
and was developed further by Eisenhart \cite{Eis}. In the particular case of the sphere $S^n$, which we will be interested in,
the first (and the most important) example, that of elliptic coordinates, was introduced already in 1859 by C.~Neumann \cite{Neu}.
The classification problem of all separation coordinates on $S^n$ had been studied in detail
by Olevski \cite{Olev} and Kalnins \& Miller \cite{KM86}.
In particular, in the last paper to describe the answer the authors used a sophisticated graphical procedure
similar to the one developed by Vilenkin to describe the polyspherical coordinates \cite{Vil}.

Despite all these advances almost nothing has been known about the global
geometry of the space of separation coordinates.  In the present paper we fill
this gap by describing the topology and algebraic geometry of this space in
the case of $S^n$. In particular, we link the graphical procedures from
\cite{KM86} and \cite{Vil} with the rich combinatorial theory of {\it
associahedra}, or {\it Stasheff polytopes}, introduced by Stasheff in 1963 in
homotopy theory \cite{Stas}.

The reason this has never been done before is probably that the algebraic
equations involved seemed far too complicated for a direct solution.
Interestingly, St\"ackel already commented on this in his habilitation thesis
from 1891, where he notes \cite[p.~6]{Staeckel}:
\begin{quote}
	\it
	\selectlanguage{german}
	Die Diskussion dieser Gleichungen ergab, dass es f\"ur $n=2$ drei wesentlich
	verschiedene Formen der Gleichung $H=0$ giebt, bei denen diese notwendigen
	Bedingungen erf\"ullt sind, und da diese Gleichungen auch wirklich
	Separation der Variabeln gestatten, ist die Frage f\"ur den Fall $n=2$
	vollst\"andig erledigt.  Aber schon f\"ur $n=3$ werden die algebraischen
	Rechnungen so umst\"andlich, dass mir eine weitere Verfolgung dieses Weges
	aussichtslos erschien.
	\selectlanguage{english}
	\footnote{
		``For $n=2$ the discussion of these equations yielded three essentially
		different forms of the equation $H=0$ [the Hamilton-Jacobi equation],
		for which the necessary conditions are satisfied, and since these
		equations indeed allow a separation of variables, the question is
		completely settled in the case $n=2$.  However, already for $n=3$ the
		algebraic computations become so cumbersome, that it seemed hopeless
		to me to pursue this approach further.''
	}
\end{quote}
In a sense, our goal in this paper is to accomplish St\"ackel's computations
for arbitrary $n$, using a substantial progress made in the theory of moduli
spaces in the last few decades.  More precisely, we prove the following
result.
\begin{mthm}
	The St\"ackel systems on $S^n$ with diagonal algebraic
	curvature tensor form a smooth projective variety isomorphic to
	the real Deligne-Mumford-Knudsen moduli space $\bar M_{0,n+2}(\mathbb R)$
	of stable genus zero curves with $n+2$ marked points.
\end{mthm}
As a corollary we have that the set $X_n$ of equivalence classes of separation
coordinates on the sphere $S^n$ modulo the orthogonal group is in one-to-one
correspondence with the quotient space $Y_n=\bar M_{0,n+2}(\mathbb
R)/S_{n+1}$. Since the real version of the Deligne-Mumford-Knudsen moduli
space $\bar M_{0,n+2}(\mathbb R)$ is known to be tessellated by $(n+1)!/2$
copies of the Stasheff polytope $K_n$ after Kapranov \cite{Kap} and Devadoss \cite{Devadoss},
we can use the known
results about $\bar M_{0,n+2}(\mathbb R)$ and associahedra \cite{DR} to
describe the combinatorial structure of $X_n$. In particular, we use the
mosaic operad \cite{Devadoss} to give an explicit construction for St\"ackel
systems and separation coordinates.

Note that in this way we establish (and exploit) a surprising correspondence
between two seemingly completely unrelated objects -- separation coordinates
on a sphere on one hand and stable genus zero curves with marked points on the
other hand -- revealing yet another guise of the famous moduli space $\bar
M_{0,n}(\mathbb R)$.

The algebraic nature of the problem of separation of variables was explicitly revealed in \cite{Schoebel1}, where the Nijenhuis integrability conditions for Killing tensors were reduced to purely algebraic equations for the associated algebraic curvature tensors.  The statement of the Main Theorem stems from a thorough analysis of these equations in the first non-trivial case $n=3$, done in \cite{Schoebel}. Its proof is based on the recent work by Aguirre, Felder and the second author \cite{AFV}, where the moduli space $\bar M_{0,n+2}$ was identified with the set of the Gaudin subalgebras in the Kohno-Drinfeld Lie
algebra $\mathfrak t_{n+1}$. We show that the Killing tensors on $S^n$ with
diagonal algebraic curvature tensor satisfy the defining relations of the
Kohno-Drinfeld Lie algebra, which allows us to make the link with the main
result of \cite{AFV}.

\section{Separation coordinates, Killing tensors and St\"ackel systems}

\subsection{Separation coordinates}

Recall that the Hamilton-Jacobi equation
\[
	\tfrac12g^{ij}\frac{\partial W}{\partial x^i}\frac{\partial W}{\partial x^j}=E
\]
\dfn{separates} in local coordinates $x^1,\ldots,x^n$ on a Riemannian manifold
$M^n$ if it admits a solution of the form
\[
	W(x^1,\ldots,x^n;\underline c)=W_1(x^1;\underline c)+\ldots+W_n(x^n;\underline c),
	\qquad
	\det\left(\frac{\partial^2W}{\partial x^i\partial c_j}\right)\not=0,
\]
depending on $n$ parameters $\underline c=(c_1,\ldots,c_n)$.  Note that if we
reparametrise each coordinate $x^i$ with a strictly monotonic function
$\Phi_i$, the Hamilton-Jacobi equation is still separable in the new
coordinates $\Phi_i(x^i)$.  The same is true for a permutation of the
variables.  In order to avoid this arbitrariness, we consider different
coordinate systems as equivalent if they are related by such transformations.
By abuse of language we will call a corresponding equivalence class simply
\dfn{separation coordinates}.  Equivalently, we can think of separation
coordinates as the (unordered) system of coordinate hypersurfaces defined by
the equations $x^i=\text{constant}$.  The separation coordinates are called
\dfn{orthogonal}, if the normals of these hypersurfaces are mutually
orthogonal.

The main tool in studying orthogonal separation coordinates are Killing
tensors that satisfy a certain condition.  The details of this relation will
be explained in the rest of this section, with emphasis on spheres.

\subsection{Killing tensors}

\begin{definition}
	\label{def:Killing}
	A \dfn{Killing tensor} on a Riemannian manifold $(M,g)$ is an element
	$K\in\Gamma(S^2T^*\!M)$ satisfying, in any coordinate system $x^{\alpha}$
	($\alpha=1,\dots,n$), the equation
	\begin{equation}
		\label{eq:Killing}
		\nabla_{\alpha}K_{\beta\gamma}+\nabla_{\beta}K_{\gamma\alpha}+\nabla_{\gamma}K_{\alpha\beta}=0,
	\end{equation}
	where $\nabla$ is the Levi-Civita connection of the metric $g$.
\end{definition}
Note that the metric $g$ is trivially a Killing tensor, because it is
covariantly constant: $\nabla_{\alpha}g_{\beta\gamma}=0.$  Here we will be concerned with Killing tensors on the
standard round sphere $S^n$, regarded as the hypersurface
\[
	S^n=\{x\in V\colon\;\lVert x\rVert=1\}\subset V
\]
of unit vectors in an $(n+1)$-dimensional Euclidean vector space $V$, equipped
with the induced metric $g$.
\begin{definition}
	An \dfn{algebraic curvature tensor} on a vector space $V$ is an element
	$R\in(V^*)^{\otimes4}$ satisfying the usual (algebraic) symmetries of a
	Riemannian curvature tensor, namely:
	\begin{subequations}
		\label{eq:R}
		\begin{align}
			\label{eq:R:anti}   R_{jikl}&=-R_{ijkl}=R_{ijlk} &&\text{(antisymmetry)}\\
			\label{eq:R:pair}   R_{klij}&=R_{ijkl}           &&\text{(pair symmetry)}&\\
			\label{eq:R:Bianchi}R_{ijkl}&+R_{iklj}+R_{iljk}=0&&\text{(Bianchi identity)}.
		\end{align}
	\end{subequations}
\end{definition}
The space of Killing tensors $K$ on $S^n\subset V$ is naturally isomorphic to the space
of algebraic curvature tensors $R$ on $V$ \cite{McLMS}.  This isomorphism is
explicitly given by the formula
\begin{equation}
	\label{eq:correspondence}
	K_x(v,w):=R(x,v,x,w)=\mspace{-8mu}\sum_{i,j,k,l=1}^{n+1}\mspace{-8mu}R_{ijkl}x^ix^kv^jw^l,
	\qquad
	x\in S^n,
	\quad
	v,w\in T_xS^n,
\end{equation}
where we consider a point $x\in S^n$ as well as the tangent vectors $v,w\in
T_xS^n$ as vectors in $V$ satisfying $\lVert x\rVert=1$ and $v,w\perp x$.  The
above isomorphism is equivariant under the natural actions of the isometry
group $\OG(V)$ on Killing tensors and on algebraic curvature tensors respectively.

\subsection{St\"ackel systems}

In Definition~\ref{def:Killing}, a Killing tensor is a symmetric bilinear form
$K_{\alpha\beta}$ on the manifold $M$.  In what follows we will interpret it
in two other ways, each of which gives rise to a Lie bracket and hence to a
Lie algebra generated by Killing tensors.  On one hand, we can use the metric
to identify the symmetric bilinear form $K_{\alpha\beta}$ with a symmetric
endomorphism ${K^\alpha}_\beta$.  Interpreted in this way, the space of
Killing tensors generates a Lie subalgebra of $\Gamma(\mathrm{End}(TM))$ with
respect to the commutator bracket
\begin{equation}
	\label{eq:commutator}
	[K,L]=KL-LK.
\end{equation}
On the other hand, we can interpret a Killing tensor $K_{\alpha\beta}$ as a
function $K_{\alpha\beta}p^\alpha p^\beta$ on the total space of the cotangent
bundle $T^*\!M$ which is quadratic in the fibres.  Interpreted in this way,
the space of Killing tensors generates a Lie subalgebra of $C^\infty(T^*\!M)$
with respect to the Poisson bracket
\begin{equation}
	\label{eq:Poisson}
	\{K,L\}
	=\sum_{\alpha=1}^n
	\left(
		\frac{\partial K}{\partial x^\alpha}
		\frac{\partial L}{\partial p_\alpha}
		-
		\frac{\partial L}{\partial x^\alpha}
		\frac{\partial K}{\partial p_\alpha}
	\right).
\end{equation}
\begin{definition}
	\label{def:Staeckel}
	A \dfn{St\"ackel system} on an $n$-dimensional Riemannian manifold is an
	$n$-dimensional space of Killing tensors which mutually commute with
	respect to both of the following brackets:
	\begin{enumerate}
		\item \label{it:commutator} the commutator bracket \eqref{eq:commutator}
		\item \label{it:Poisson}    the Poisson bracket \eqref{eq:Poisson}
	\end{enumerate}
\end{definition}

\begin{rem}
	In the initial definition given by St\"ackel, condition (\ref{it:Poisson})
	is replaced by an integrability condition on the eigen spaces of the
	Killing tensors \cite{Staeckel}.  The equivalence of both definitions is
	proven in \cite{Ben}.
\end{rem}

It can be shown that every St\"ackel system contains a Killing tensor with
simple eigenvalues \cite{Ben}.  Moreover, the $n$ distributions given by the
orthogonal complements of its eigendirections are integrable.  Hence they
define $n$ hypersurface foliations with orthogonal normals or, equivalently,
orthogonal coordinates.  On the other hand, every Killing tensor commuting
with the above and having simple eigenvalues defines the same coordinates.  In
this manner each St\"ackel system defines a unique coordinate system.  It is a
classical result that these are separation coordinates and that every system
of orthogonal separation coordinates arises in this way from a St\"ackel system:
\begin{thm}[\cite{Staeckel,Eis,Ben}]
	\label{thm:correspondence}
	There is a bijective correspondence between St\"ackel systems and
	orthogonal separation coordinates.
\end{thm}
\begin{rem}
	A priori, the above result is only a local result.  However, any local
	Killing tensor field on a sphere can be extended to a global one (see
	e.\,g.\ \cite{BKW}).  Hence the same is true for the corresponding
	separation coordinates.  That is why we can use the above result for a
	global classification of orthogonal separation coordinates on $S^n$.
\end{rem}

\subsection{Killing tensors with diagonal algebraic curvature tensor}

In the following we will only consider Killing tensors on $S^n$ whose
algebraic curvature tensor is diagonal in the following sense.
\begin{definition}
	Due to the symmetries \eqref{eq:R:anti} and \eqref{eq:R:pair}, we can
	interpret an algebraic curvature tensor $R$ on $V$ as a symmetric bilinear
	form on $\Lambda^2V$.  We say that $R$ is \dfn{diagonal} in an orthonormal
	basis $\{e_i:1\le i\le n+1\}$ of $V$, if it is diagonal as a bilinear form
	on $\Lambda^2V$ in the associated basis $\{e_i\wedge e_j:1\le i<j\le
	n+1\}$.  In components, this simply means that $R_{ijkl}=0$ unless
	$\{i,j\}=\{k,l\}$.
\end{definition}
Restricting to Killing tensors with diagonal algebraic curvature tensor does
not mean any loss of generality.  The reason is the following refinement of
Theorem~\ref{thm:correspondence} for spheres.
\begin{thm}[\cite{BKW}]
	\label{thm:diagonalisability}
	Necessary and sufficient conditions for the existence of an orthogonal
	separable coordinate system for the Hamilton-Jacobi equation on $S^n$ are
	that there are $n$ Killing tensors with diagonal algebraic curvature
	tensor, one of which is the metric, which are linearly independent
	(locally) and pairwise commute with respect to the Poisson bracket.
\end{thm}
\begin{rem}
	\label{rem:commutator-equivalence}
	By this theorem condition (\ref{it:Poisson}) in
	Definition~\ref{def:Staeckel} implies condition (\ref{it:commutator}) on
	$S^n$.  As we will show in Section~\ref{sec:correspondence} below, both
	conditions are actually equivalent for $S^n$.
\end{rem}

The restriction to separation coordinates which are orthogonal does not
constitute a loss of generality either, because of the following result.
\begin{thm}[\cite{KM86}]
	\label{thm:orthogonal}
	All separation coordinates on $S^n$ are equivalent to
	orthogonal separation coordinates.
\end{thm}
The equivalence corresponds to a linear change of the so-called ignorable
coordinates, on which the metric does not depend (see \cite{KM86}). We will
consider separation coordinates up to this equivalence throughout this paper.

Theorems~\ref{thm:diagonalisability} and \ref{thm:orthogonal} reduce the
classification of separation coordinates on spheres to the purely algebraic
problem of finding certain abelian subalgebras in the following two Lie
algebras.
\begin{definition}
	\label{def:d_n}
	We denote by $\mathfrak d_{n+1}\subset\Gamma(\mathrm{End}(TS^n))$ and by
	$\mathscr D_{n+1}\subset C^\infty(T^*\!S^n)$ the Lie subalgebras generated
	by Killing tensors with diagonal algebraic curvature tensor under the
	commutator bracket \eqref{eq:commutator} and the Poisson
	bracket \eqref{eq:Poisson} respectively.
\end{definition}
By definition, a diagonal algebraic curvature tensor $R$ is uniquely
determined by the diagonal elements $R_{ijij}$ for $1\le i<j\le n+1$.  Indeed,
the symmetries \eqref{eq:R:anti} and \eqref{eq:R:pair} determine the
components $R_{ijij}=-R_{ijji}=R_{jiji}=-R_{jiij}$ for $i<j$.  And if we set
all other components to zero, the resulting tensor $R$ satisfies all
symmetries \eqref{eq:R} of an algebraic curvature tensor.  For fixed $i$ and
$j$, let $K_{ij}$ be the Killing tensor with diagonal algebraic curvature
tensor $R$ given by
\begin{equation}
	\label{eq:R:diag}
	R_{ijij}=-R_{ijji}=R_{jiji}=-R_{jiij}=1
\end{equation}
and all other components zero.  Then $K_{ij}$ with $i<j$ form a basis of
the space of Killing tensors on $S^n$ with diagonal algebraic curvature
tensor and constitute a set of generators for both Lie
algebras, $\mathfrak d_{n+1}$ and $\mathscr D_{n+1}$.  The following two
propositions show that they satisfy the same relations in $\mathfrak d_{n+1}$
and in $\mathscr D_{n+1}$.
\begin{prop}
	\label{prop:commutator:relations}
	Let $K_{ij}$, $1\le i<j\le n+1$ be the basis of the space of Killing tensors on $S^n$
	with diagonal algebraic curvature tensor, as defined above.  For
	convenience we set $K_{ji}:=K_{ij}$.  Then $K_{ij}$ satisfy the
	following relations in $\mathfrak d_{n+1}$:
	\begin{subequations}
		\label{eq:commutator:relations}
		\begin{align}
			\label{eq:commutator:relations:disjoint}
			[K_{ij},K_{kl}]&=0&&\text{if $i,j,k,l$ are distinct}\\
			\label{eq:commutator:relations:overlapping}
			[K_{ij},K_{ik}+K_{jk}]&=0&&\text{if $i,j,k$ are distinct}.
		\end{align}
	\end{subequations}
\end{prop}
\begin{proof}
	We can extend the Killing tensor $K$ on $T_xS^n$ to a symmetric tensor
	$\hat K$ on $V=T_xS^n\oplus\mathbb Rx$ by omitting the restriction
	$v,w\perp x$ in \eqref{eq:correspondence}.  The antisymmetry \eqref{eq:R:anti} implies
	that $K_x(v,x)=0$ for any $v\in V$, so $\hat K$ is the extension of $K$ by
	zero.  Consequently, we have $[\hat K_{ij},\hat
	K_{kl}]=\widehat{[K_{ij},K_{kl}]}$, so that it is sufficient to check the
	above relations on the corresponding extensions.  To do so, consider the
	Killing tensor $K_{ij}$ at a point $x\in S^n$.  By
	\eqref{eq:correspondence} and the definition \eqref{eq:R:diag} of the
	diagonal algebraic curvature tensor of $K_{ij}$ we have
	\begin{align*}
		\hat K_{ij}(v,w)
		&=\sum_{a,b=1}^{n+1}(R_{abab}x^ax^av^bw^b+R_{abba}x^ax^bv^bw^a)\\
		&=x^ix^iv^jw^j+x^jx^jv^iw^i-x^ix^jv^iw^j-x^ix^jv^jw^i.
	\end{align*}
	Let us put all indices down for convenience.
	Then we have
	\[
		\hat K_{ij}=
		\begin{pmatrix}
			x_j^2&-x_ix_j\\
			-x_ix_j&x_i^2
		\end{pmatrix},
	\]
	where we left only non-zero ($i$-th and $j$-th) rows
	and columns.  This already proves relation
	\eqref{eq:commutator:relations:disjoint}.  To check the remaining
	relation \eqref{eq:commutator:relations:overlapping}, we compute
	\begin{equation}
		\label{eq:commutator:KijKjk}
		\begin{split}
			[\hat K_{ij},\hat K_{jk}]
			&=
			\left[
				\begin{pmatrix}
					x_j^2&-x_ix_j&0\\
					-x_ix_j&x_i^2&0\\
					0&0&0
				\end{pmatrix},
				\begin{pmatrix}
					0&0&0\\
					0&x_k^2&-x_jx_k\\
					0&-x_jx_k&x_j^2
				\end{pmatrix}
			\right]\\
			&=
			x_ix_jx_k
			\begin{pmatrix}
				0&-x_k&x_j\\
				x_k&0&-x_i\\
				-x_j&x_i&0
			\end{pmatrix}
		\end{split}.
	\end{equation}
	Here we omitted rows and columns other than $i,j,k$, because they are
	zero.  In the same way we compute $[\hat K_{ij},\hat K_{ik}]$ and verify
	\eqref{eq:commutator:relations:overlapping}.
\end{proof}
\begin{prop}
	\label{prop:Poisson:relations}
	As elements of $\mathscr D_{n+1}$ the generators $K_{ij}$
	satisfy the following relations:
	\begin{subequations}
		\label{eq:Poisson:relations}
		\begin{align}
			\label{eq:Poisson:relations:disjoint}
			\{K_{ij},K_{kl}\}&=0&&\text{if $i,j,k,l$ are distinct}\\
			\label{eq:Poisson:relations:overlapping}
			\{K_{ij},K_{ik}+K_{jk}\}&=0&&\text{if $i,j,k$ are distinct}.
		\end{align}
	\end{subequations}
\end{prop}
\begin{proof}
		The function on $T^*S^n$ given by $K_{ij}$ is
	\begin{subequations}
		\begin{equation}
			\label{eq:Poisson:Kij}
			K_{ij}(x,p)=x_j^2p_i^2+x_i^2p_j^2-2x_ix_jp_ip_j=(x_ip_j-x_jp_i)^2.
		\end{equation}
		This already proves relation \eqref{eq:Poisson:relations:disjoint}.  In order to
		verify relation \eqref{eq:Poisson:relations:overlapping}, we compute
		\begin{equation}
			\label{eq:Poisson:KijKjk}
			\begin{split}
				\{K_{ij},K_{jk}\}
				&=
				\frac{\partial K_{ij}}{\partial x_j}
				\frac{\partial K_{jk}}{\partial p_j}
				-
				\frac{\partial K_{ij}}{\partial p_j}
				\frac{\partial K_{jk}}{\partial x_j}\\
				&=4
				(x_ip_j-x_jp_i)
				(x_jp_k-x_kp_j)
				(x_kp_i-x_ip_k),
			\end{split}
		\end{equation}
	\end{subequations}
	which is clearly antisymmetric with respect to $i$ and $j$.
\end{proof}
The next proposition says that there are no more relations between the
generators and their brackets, both in $\mathfrak d_{n+1}$ as well as in
$\mathscr D_{n+1}$, provided $n\ge 2$.
\begin{prop}
	\label{prop:commutator:independence}
	The generators $K_{ij}$ and the commutator brackets $[K_{ij},K_{jk}]$ with
	$1\le i<j<k\le n+1$ are linearly independent in $\mathfrak d_{n+1}$ for
	$n\ge 2$.  The same is true for $K_{ij}$ and the Poisson brackets
	$\{K_{ij},K_{jk}\}$ as elements of $\mathscr D_{n+1}$.
\end{prop}
\begin{proof}
	Since $K_{ij}$ are symmetric and $[K_{ij},K_{jk}]$ are
	antisymmetric, it suffices to check the linear independence of both sets
	independently.  The elements $K_{ij}$ are linearly independent by definition. To prove
	the linear independence of $[K_{ij},K_{jk}]$ suppose that
	\[
		\sum_{1\le i<j<k\le n+1}\mspace{-16mu}\lambda_{ijk}[K_{ij},K_{jk}]=0.
	\]
	For each triple $(p,q,r)$ with $1\le p<q<r\le n+1$ consider a point $x\in
	S^n$ with $x_m=0$ if and only if $m\not\in\{p,q,r\}$.  Due to
	\eqref{eq:commutator:KijKjk} we have that $[K_{ij},K_{jk}]=0$ at this
	point $x$ for all $i<j<k$ unless $(i,j,k)=(p,q,r)$.  Hence
	all $\lambda_{pqr}=0$.

	In the case of $\mathscr D_{n+1}$ we note that $K_{ij}$ are quadratic in momenta while $\{K_{ij},K_{jk}\}$
	are cubic; the rest of the proof is the same.
\end{proof}

\subsection{St\"ackel systems generated from special conformal Killing tensors}

The generic St\"ackel systems on a sphere can be constructed from special
conformal Killing tensors.
\begin{definition}
	A \dfn{special conformal Killing tensor} on a Riemannian manifold is a
	symmetric tensor $L_{\alpha\beta}$ satisfying
	\begin{align*}
		\label{eq:SCKT}
		\nabla_{\gamma}L_{\alpha\beta}&=\lambda_\alpha g_{\beta\gamma}+\lambda_\beta g_{\alpha\gamma}&
		\lambda&=\tfrac12\nabla\operatorname{tr}L.
	\end{align*}
\end{definition}
The space of special conformal Killing tensors parametrises geodesically
equivalent metrics, which are metrics having the same set of unparametrised geodesics.
Their importance in our context stems from the fact that
\[
	K:=L-(\operatorname{tr}L)g
\]
defines a Killing tensor, as one immediately checks, and the fact that every
Killing tensor of this form is contained in a St\"ackel system.  The latter
follows easily from the Nijenhuis integrability conditions applied to $K$ (see for
example \cite{Schoebel}).  Thus, in the generic case where $L$ (and hence $K$)
has pairwise different eigenvalues, it defines a system of separation
coordinates.  The corresponding St\"ackel system is spanned by the
coefficients of the polynomial
\begin{equation}
	\label{eq:Benenti}
	\operatorname{Adj}(L-\lambda\mathrm{Id})=\sum_{i=0}^{n-1}K_i\lambda^i,
\end{equation}
where $\operatorname{Adj}X$ denotes the adjugate matrix, i.\,e.\ the transpose of the
cofactor matrix of $X$ \cite{Bolsinov&Matveev}.

We can deduce the corresponding separation coordinates directly from the
special conformal Killing tensor $L$, since its eigenvalues are constant on
the corresponding coordinate hypersurfaces \cite{Crampin03a}.  This means that
the eigenvalues of $L$ can be taken as coordinate functions.  On $S^n\subset
V$ the situation is further simplified.  The reason is that under certain
conditions, which are met in this case, every special conformal Killing tensor
$L$ on $M$ is the restriction of a covariantly constant symmetric tensor $\hat
L$ on the metric cone over $M$ and vice versa (see for example \cite{MM}).
Here the metric cone over $S^n\subset V$ is nothing but $V$, so the
determination of separation coordinates on $S^n\subset V$ arising from a
special conformal Killing tensor reduces to computing the eigenvalues of the
restriction $L$ of a constant symmetric tensor $\hat L$ on $V$.

\subsection{Two extremal cases:  elliptic and polyspherical coordinates}

As a matter of example, let us consider two extremal cases.  The generic case
of orthogonal separation coordinates on the sphere consists of elliptic
coordinates and can be obtained from a special conformal Killing tensor $L$
with simple eigenvalues as described above. For $\hat L$ with (constant)
eigenvalues $\Lambda_1<\Lambda_2<\ldots<\Lambda_{n+1}$ the eigenvalues
$\lambda_1(x),\ldots,\lambda_n(x)$ of $L$ at a point
$x=(x_1,\ldots,x_{n+1})\in S^n$ are the solutions of the equation
\begin{equation}
	\label{eq:eigenvalues}
	\sum_{k=1}^{n+1}\frac{x_k^2}{\Lambda_k-\lambda}=0
	\qquad
	\lVert x\rVert^2=1
\end{equation}
which can be ordered to satisfy
\[
	\Lambda_1<\lambda_1(x)<\Lambda_2<\lambda_2(x)<\cdots<\lambda_n(x)<\Lambda_{n+1}.
\]
This is nothing else but the defining equation for the classical {\it elliptic
coordinates} on the sphere $S^n$ introduced in 1859 by C.~Neumann \cite{Neu}.
Note that shifting or multiplying the parameters
$\Lambda_1<\Lambda_2<\ldots<\Lambda_{n+1}$ by a constant results in a mere
reparametrisation of the same coordinate system.  Therefore elliptic
coordinates form an $(n-1)$-parameter family of separation coordinates on
$S^n$.

The other extreme, having no continuous parameters at all, are {\it
polyspherical coordinates} considered by Vilenkin \cite{Vil,VilK}.
Each of these coordinate systems is given in terms of Cartesian coordinates by
starting with $\boldsymbol x(\varnothing):=1$ on $S^0\subset\mathbb R^1$ and
then defining recursively $\boldsymbol z=\boldsymbol
z(\varphi_1,\ldots,\varphi_{n-1})$ on $S^{n-1}\subset\mathbb R^n$ from
$\boldsymbol x=\boldsymbol x(\varphi_1,\ldots,\varphi_{n_1-1})$ on
$S^{n_1-1}\subset\mathbb R^{n_1}$ and $\boldsymbol y=\boldsymbol
y(\varphi_{n_1},\ldots,\varphi_{n_1+n_2-2})$ on $S^{n_2-1}\subset\mathbb
R^{n_2}$ by setting
\begin{equation}
	\label{eq:polyspherical}
	\boldsymbol z=(\boldsymbol x\cos\varphi_{n-1},\boldsymbol y\sin\varphi_{n-1})
\end{equation}
for $n=n_1+n_2$.  Since this involves a choice of a splitting $n=n_1+n_2$ in
each step, polyspherical coordinates on $S^{n-1}$ are parametrised by planar
rooted binary trees with $n$ leaves.%
\footnote{When Vilenkin introduced polyspherical coordinates in \cite{Vil}, he
	used trees which are not binary.  The description with binary trees
	appeared in Vilenkin and Klimyk \cite[Chap.~10.5]{VilK} and both are
	completely equivalent.}
For example, the standard spherical coordinates correspond to the binary tree
where each right child is a leaf.

\subsection{The residual action of the isometry group}

Theorem~\ref{thm:diagonalisability} implies that for a St\"ackel system on a
sphere we can always find an isometry which takes all Killing tensors in this
St\"ackel system to Killing tensors having diagonal algebraic curvature tensors.
This means that the space of Killing tensors with diagonal algebraic curvature tensor
defines a slice for the action of the isometry group.  If we want to classify
separation coordinates up to isometries, we have to take into account that the
stabiliser of this slice in the isometry group is not trivial.

Due to the symmetries \eqref{eq:R:anti} and \eqref{eq:R:pair}, the space of
algebraic curvature tensors is a subspace of the space $S^2\Lambda^2V$ of
symmetric forms on $\Lambda^2V$.  The natural action of the isometry group
$\OG(V)$ on this space is given as follows.  Mapping an orthonormal basis
$\{e_i:\,1\le i\le n+1\}$ of $V$ to the basis $\{(e_i\wedge e_j)/\sqrt2:\,1\le
i<j\le n+1\}$ of $\Lambda^2V$ defines a map
\begin{equation}
	\label{eq:map}
	\OG(V)\to\OG(\Lambda^2V),
\end{equation}
since the latter basis is orthonormal with respect to the scalar product on
$\Lambda^2V$ induced from the one on $V$.  Via the action of $\OG(\Lambda^2V)$
on $\Lambda^2V$ this defines an action of $\OG(V)$ on $S^2\Lambda^2V$ and
hence on algebraic curvature tensors.

In general, the subgroup in $\OG(V)$ leaving the space of diagonal bilinear
forms on $V$ invariant is the subgroup of signed permutation matrices, acting
by permutations and sign changes of the chosen basis $\{e_i\}$ in $V$.  This
group is the symmetry group of the hyperoctahedron in $V$ with vertices $\pm
e_i$ and is isomorphic to the semidirect product $S_N\ltimes\mathbb Z_2^N$,
where $N=\dim V=n+1$.

The stabiliser in $\OG(V)$ of the space of diagonal algebraic curvature
tensors is now the preimage under \eqref{eq:map} of the stabiliser of diagonal
bilinear forms on $\Lambda^2V$.  Since the latter consists of permutations and
sign changes of the basis $\{e_i\wedge e_j\}$ in $\Lambda^2V$, this is just
the group of permutations and sign changes of the basis elements $e_i$,
i.\,e.\ the group described in the preceding paragraph.  Note that the normal
subgroup of sign changes, which is isomorphic to $\mathbb Z_2^N$, acts
trivially on diagonal bilinear forms.  Hence the action descends to the
quotient $(S_N\ltimes\mathbb Z_2^N)/\mathbb Z_2^N\cong S_N$.  Summarising the
above, we have:
\begin{prop}
	\label{prop:residual}
	The stabiliser in the isometry group $\OG(V)$ of the space of Killing
	tensors on $S^n\subset V$ with diagonal algebraic curvature tensor is the
	hyperoctahedral group and isomorphic to the semidirect product
	$S_N\ltimes\mathbb Z_2^N$, where $N=n+1$.  This action descends to an
	action of $S_N$ given by
	\begin{equation}
		\label{eq:groupaction}
		\sigma(K_{ij})=K_{\sigma(i)\sigma(j)},
		\qquad
		\sigma\in S_N.
	\end{equation}
\end{prop}

\section{Gaudin subalgebras of the Kohno-Drinfeld Lie algebra and the moduli space \texorpdfstring{$\bar M_{0,n+1}$}{}.}

We describe now the result of \cite{AFV}, which plays a key role for us.
The (real version of) the {\it Kohno-Drinfeld Lie algebra} $\mathfrak t_n$ ($n=2,3,\dots$)
 is defined as the quotient of the free Lie
algebra over $\mathbb R$ with generators $t_{ij}=t_{ji}$, $i\neq j\in\{1,\dots,n\}$, by
the ideal generated by the relations
\begin{eqnarray}
\label{KDrel}
{}[t_{ij},t_{kl}]&=&0, \qquad \text{if $i,j,k,l$ are distinct,}\\
{} [t_{ij},t_{ik}+t_{jk}]&=&0, \qquad  \text{if $i,j,k$ are
  distinct.}
\end{eqnarray}
This Lie algebra appeared in Kohno's work as the holonomy Lie algebra
of the complement to the union of the diagonals $z_i=z_j$, $i<j$ in
$\mathbb C^n$ (which is also a configuration space of $n$ distinct points on the complex plane)
and in Drinfeld's work as the value space of the universal Knizhnik-Zamolodchikov connection
(see the references in \cite{AFV}).

{\it Gaudin subalgebras} of Kohno-Drinfeld Lie algebras were introduced in \cite{AFV} as
the abelian Lie subalgebras of maximal
dimension contained in the linear span $\mathfrak t^1_n$ of the
generators $t_{ij}$.
The main class of examples is
provided by Gaudin's models of integrable spin chains
\begin{equation}\label{e-0}
\mathfrak g_n(z)=\left\{\sum_{1\le i<j\le
n}\frac{b_i-b_j}{z_i-z_j}\,t_{ij},\quad b\in \mathbb R^n\right\}.
\end{equation}
Note that they are parametrised by $z\in\Sigma_n/\mathrm{Aff}$, where
\[
\Sigma_n=\mathbb R^n\smallsetminus\bigcup_{i<j}\{z\in\mathbb R^n\,|\,z_i=z_j\}
\]
is the configuration space of $n$ distinct ordered points on the real line and
$\mathrm{Aff}$ is the group of affine maps $z\mapsto az+b$, $a\neq0$, acting
diagonally on $\mathbb R^n$.  A different type of example, which came from the
representation theory of the symmetric group, is given by the {\it
Jucys--Murphy subalgebras} spanned by
\begin{equation}
	\label{eq:Jucys-Murphy}
	t_{12},\quad t_{13}+t_{23},\quad t_{14}+t_{24}+t_{34},\quad\dots
\end{equation}
(see \cite{VershikOkounkov} and references therein).

The main result of \cite{AFV} is the following.
\begin{thm}\cite{AFV}\label{t-1}
Gaudin subalgebras in $\mathfrak t_n$ form a nonsingular algebraic subvariety
of the Grassmannian $G(n-1,n(n-1)/2)$ of $(n-1)$-dimensional subspaces in
$\mathfrak t_n^1$, isomorphic to the moduli space $\bar M_{0,n+1}$ of
stable curves of genus zero with $n+1$ marked points.
\end{thm}

In fact, the result holds for any quotient of $\mathfrak t_n$ where both the generators $t_{ij}$, $1\le i<j\le n$, and
the brackets $[t_{ij},t_{jk}]$, $1\le i<j<k\le n$, are linearly independent (see remark 2.6 in \cite{AFV}), and over any field.

The most popular version of the moduli space $\bar M_{0,n+1}$ -- appearing, for
example, in the celebrated Witten's conjecture -- is defined over $\mathbb C$.
It is a particular (Deligne-Mumford) compactification of the configuration
space $M_{0,n+1}(\mathbb C)$ of $n+1$ distinct labelled points in ${\mathbb
C}P^1$ modulo $PGL_2(\mathbb C)$ studied by Knudsen \cite{Knudsen}, who proved
that it is a smooth projective variety. The compactification $\bar
M_{0,n+1}(\mathbb C)$ includes the singular rational curves with double point
singularities and with the following properties: the graph of components is a
tree (genus zero) and each irreducible component contains at least three
marked or singular points (stability condition).

However, we need the real version $\bar M_{0,n+1}(\mathbb R)$, which we discuss next.

\section{The real version \texorpdfstring{$\bar M_{0,n+1}(\mathbb R)$}{} and Stasheff polytopes}

\subsection{Topology}

The real version $\bar M_{0,n+1}(\mathbb R)$ was studied in more detail by Kapranov \cite{Kap} and Devadoss \cite{Devadoss}.
By Knudsen's theorem, which works over $\mathbb R$ as well, $\bar M_{0,n+1}(\mathbb R)$ is a smooth real manifold of dimension $n-2$. It can be described as an iterated blow-up of $\mathbb RP^{n-2}$ \cite{Kap, Devadoss, DJS}.
$\bar M_{0,4}(\mathbb R)$ is simply $\mathbb RP^{1}$ and $\bar M_{0,5}(\mathbb R)$ is a non-orientable surface with Euler characteristic $-3$, which is a connected sum of five copies of $\mathbb RP^{2}$.

The topology of $\bar M_{0,n+1}(\mathbb R)$ becomes increasingly complicated when $n$ grows. It is known to be aspherical (Davis et al.\ \cite{DJS}). The Euler characteristic can be given explicitly by
$$
\chi\bigl(\bar M_{0,n+1}(\mathbb R)\bigr)=(-1)^{\frac{n-2}{2}}(n-1)!!(n-3)!!
$$
for even $n$ (and zero for odd $n$), see \cite{Devadoss}. A description of the cohomology is more complicated than in the complex case, as found by Etingof et al.\ in \cite{ER}.

\subsection{Combinatorics}

Fortunately, a lot of information about $\bar M_{0,n+1}(\mathbb R)$ is encapsulated in a well studied remarkable polytope known as {\it associahedron}, or {\it Stasheff polytope} $K_{n}$. Namely, $\bar M_{0,n+1}(\mathbb R)$ is tessellated by $n!/2$ copies of $K_n$, see \cite{Kap, Devadoss}.

$K_n$ was first described by Stasheff as a combinatorial object in the homotopy theory of $H$-spaces \cite{Stas} (see the history of this in Stasheff's contribution to \cite{MPS}). Its first realisation as a convex polytope is usually ascribed to Milnor. By now we have several geometric realisations of Stasheff polytopes, see e.g.\ \cite{Dev2} and references therein. $K_n$ is a convex polytope of dimension $n-2$:
$K_3$ is a segment, $K_4$ is a pentagon and $K_5$ is the polyhedron shown in Figure~\ref{stas}, which can be obtained combinatorially by cutting off three skew edges from a cube.

\begin{figure}[h]
	\subfigure[$K_3$]{\includegraphics[width=.2\textwidth]{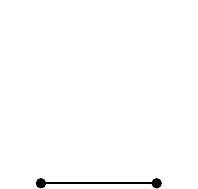}}\hfill
	\subfigure[$K_4$]{\includegraphics[width=.2\textwidth]{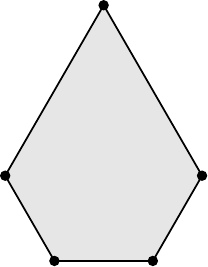}}\hfill
	\subfigure[$K_5$]{\includegraphics[width=.4\textwidth]{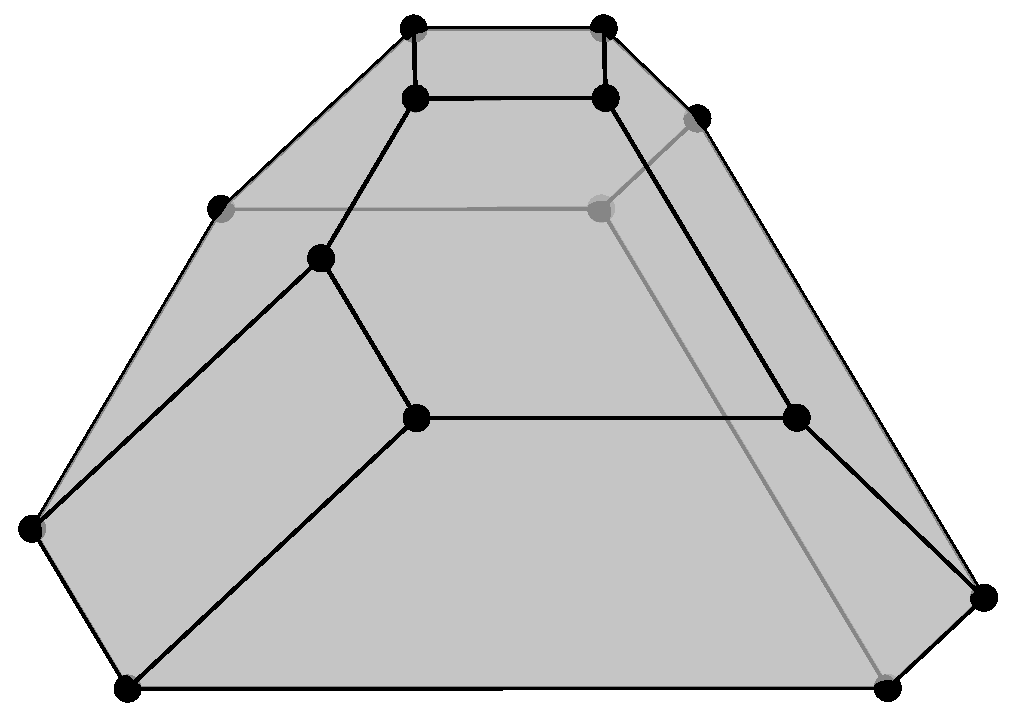}}
	\caption{Stasheff polytopes.}
	\label{stas}
\end{figure}

The faces of $K_n$ of codimension $d$ are in one-to-one correspondence with
dissections of a based $(n+1)$-gon by $d$ non-intersecting diagonals (see
e.\,g.\ \cite{DR}). In particular, the vertices of $K_n$ correspond to the
triangulations of the $(n+1)$-gon by non-intersecting diagonals and their
number is $C_{n-1}$, where
$$
C_n=\frac{1}{n+1}{{2n}\choose{n}}
$$
is the {\it Catalan number}.

Alternatively, the faces of $K_n$ can be labelled by non-isomorphic planar
rooted trees with $n$ leaves (see e.\,g.\ \cite{Devadoss}).  These are simply
the dual graphs of the dissected polygons, cut off at the edges of the polygon.  In particular, the vertices of
$K_n$ correspond to binary rooted trees.  For $n=4$ this is depicted in
Figure~\ref{fig:combinatorics}.

\begin{figure}[h]
	\begin{center}
		\includegraphics[width=.9\textwidth]{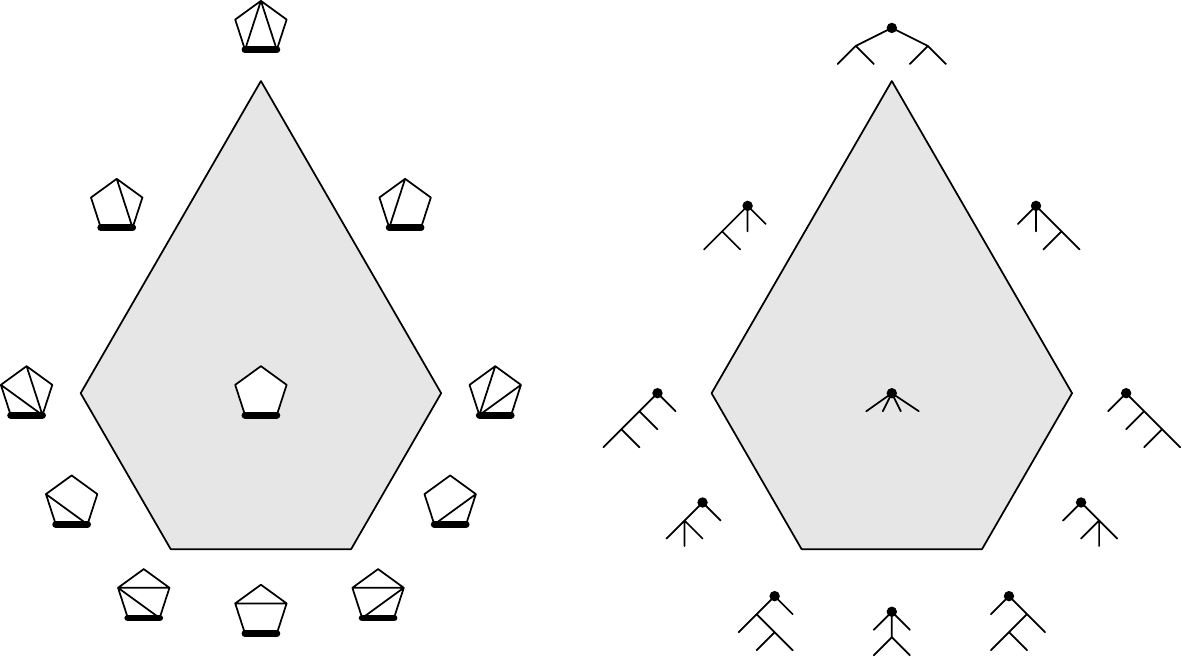}
		\caption{%
			Labellings of $K_4$ by dissections of a based pentagon (left) and
			planar rooted trees with four leaves (right).
		}
		\label{fig:combinatorics}
	\end{center}
\end{figure}

The Stasheff polytope $K_n$ admits a realisation with the dihedral symmetry
$D_{n+1}$, which is the symmetry group of a regular $(n+1)$-gon \cite{Lee}.

\subsection{Operad structure}

The sequence of moduli spaces $\bar M_{0,n+1}(\mathbb R)$ for $n=1,2,\ldots$ carries a natural
operad structure, called the ``mosaic operad'' \cite{Devadoss}.

\begin{definition}
	An \dfn{operad structure} on a sequence of objects $\mathcal O(n)$ is a
	composition map
	\[
		\begin{array}{rrcl}
			\circ:&\mathcal O(k)\times\mathcal O(n_1)\times\cdots\times\mathcal O(n_k)&\longrightarrow&\mathcal O(n_1+\cdots+n_k)\\
			&(y,x_1,\ldots,x_k)&\mapsto&y\circ(x_1,\ldots,x_k)
		\end{array}
	\]
	together with a right action
	\[
		\begin{array}{rrcl}
			\star:&\mathcal O(n)\times S_n&\longrightarrow&\mathcal O(n)\\
			&(x,\pi)&\mapsto&x\star\pi
		\end{array}
	\]
	of the permutation group $S_n$ on each object $\mathcal O(n)$, satisfying
	the following axioms:
	\begin{description}
		\item[Identity]
			There is a distinguished element $1\in\mathcal O(1)$ satisfying
			\[
				y\circ(1,\ldots,1)=y=1\circ y.
			\]
		\item[Associativity]
			\[
				z\circ(y_1\circ x_1,\ldots,y_k\circ x_k)=\bigl(z\circ(y_1,\ldots,y_k)\bigr)\circ(x_1,\ldots,x_k),
			\]
		\item[Equivariance]
			\begin{align*}
				(y\star\pi)\circ(x_1,\ldots,x_k)&=y\circ\bigl((x_1,\ldots,x_k)\star\pi\bigr)\\
				y\circ(x_1\star\pi_1,\ldots,x_k\star\pi_k)&=\bigl(y\circ(x_1,\ldots,x_k)\bigr)\star(\pi_1,\ldots,\pi_k),
			\end{align*}
			where $S_k$ acts on $(x_1,\ldots,x_k)$ by permutation and
			$(\pi_1,\ldots,\pi_k)$ on $\mathcal O(n_1+\cdots+n_k)$ under the
			inclusion $S_{n_1}\times\cdots\times S_{n_k}\hookrightarrow
			S_{n_1+\cdots+n_k}$.
	\end{description}
\end{definition}
In terms of dissected polygons, the operad structure on $\mathcal O(n):=\bar
M_{0,n+1}(\mathbb R)$ is given by gluing the $k$ $(n_i+1)$-gons $x_i$ with
their base to the $k$ non-base edges of the $(k+1)$-gon $y$ to form the
$(n_1+\ldots+n_k+1)$-gon $y\circ(x_1,\ldots,x_k)$ with the base of $y$ as
base.  If $y$ is dissected by $d_0$ diagonals and $x_i$ by $d_i$ diagonals,
then $y\circ(x_1,\ldots,x_k)$ is dissected by $d_0+d_1+\cdots d_k+k$
diagonals, namely the diagonals of $y$ and $x_1,\ldots,x_k$ plus the $k$ glued
pairs of edges which become diagonals after gluing.  On planar trees, the
composition $y\circ(x_1,\ldots,x_k)$ is given by grafting the $k$ trees $x_i$
with their root to the leaves of the tree $y$.

The operad structure on Stasheff polytopes defines a map
\[
	\circ:K_k\times K_{n_1}\times\cdots\times K_{n_k}\hookrightarrow K_{n_1+\cdots+n_k}
\]
whose image is a codimension $k$ face of $K_{n_1+\cdots+n_k}$.  This yields a
decomposition of the faces of a Stasheff polytope into products of Stasheff
polytopes \cite{Stas}.

\section{The correspondence}
\label{sec:correspondence}

After these preparations we are now in a position to state our main result.  By Propositions
\ref{prop:commutator:relations} and \ref{prop:Poisson:relations} the defining
relations of the Kohno-Drinfeld Lie algebra $\mathfrak t_n$ are satisfied in
the Lie algebras $\mathfrak d_n$ and $\mathscr D_n$ (c.\,f.\ Definition
\ref{def:d_n}).  This provides surjective Lie algebra morphisms
\begin{align}
	\label{eq:morphisms}
	\mathfrak t_n&\longrightarrow\mathfrak d_n&
	\mathfrak t_n&\longrightarrow\mathscr D_n,
\end{align}
given by mapping the generator $t_{ij}$ to the Killing tensor $K_{ij}$.  Under
these morphisms the linear span $\mathfrak t_n^1$ of the $t_{ij}$ is
isomorphic to the space of Killing tensors -- interpreted as endomorphisms in
$\mathfrak d_n$ respectively as quadratic functions on the cotangent bundle in
$\mathscr D_n$.  Thus the above morphisms map Gaudin subalgebras in the
Kohno-Drinfeld Lie algebra $\mathfrak t_n$ to St\"ackel systems on $S^{n-1}$
with diagonal algebraic curvature tensor.
Proposition~\ref{prop:commutator:independence} now shows that this defines an
isomorphism between Gaudin subalgebras and St\"ackel systems.  Now using
Theorem~\ref{t-1} we have the following correspondence:

\begin{thm}
	\label{thm:iso}
	The St\"ackel systems on $S^n$ with diagonal algebraic curvature tensor
	form a nonsingular algebraic subvariety of the Grassmannian
	$G(n,n(n+1)/2)$ of $n$-planes in the space of Killing tensors with
	diagonal algebraic curvature tensor, which is isomorphic to the real
	Deligne-Mumford-Knudsen moduli space $\bar M_{0,n+2}(\mathbb R)$ of stable
	genus zero curves with $n+2$ marked points.
\end{thm}

Note that since $\bar M_{0,n+2}(\mathbb R)$ can be considered as a
compactification of the configuration space $\Sigma_{n+1}/\mathrm{Aff}$ of
$n+1$ ordered distinct points on a real line modulo the affine group, we have
a natural action of the symmetric group $S_{n+1}$ on $\bar M_{0,n+2}(\mathbb
R)$.

\begin{cor}
	The space $X_n$ of equivalence classes of orthogonal separation
	coordinates on the sphere $S^n$ modulo the orthogonal group $O(n+1)$ is
	naturally homeomorphic to the quotient space $Y_n=\bar M_{0,n+2}(\mathbb
	R)/S_{n+1}$.
\end{cor}

\begin{proof}
	This follows directly from Theorems~\ref{thm:correspondence} and \ref{thm:iso}.
	It suffices to note that by Proposition~\ref{prop:residual} the morphisms
	\eqref{eq:morphisms} are equivariant with respect to the $S_{n+1}$-action
	on $\mathfrak t_{n+1}$ and $\mathfrak d_{n+1}$ respectively $\mathscr
	D_{n+1}$.  Therefore the isomorphism in Theorem~\ref{thm:iso} is
	$S_{n+1}$-equiviariant, so that the corresponding quotients are
	homeomorphic. Note that while the space of St\"ackel systems on
	$S^n\subset V$ with diagonal algebraic curvature tensor depends on the
	choice of an orthonormal basis in the ambient space $V$ (for which the
	algebraic curvature tensors are diagonal), the quotient does not.
	Therefore the homeomorphism is natural.
\end{proof}

Since $\bar M_{0,n+2}(\mathbb R)$ is tessellated by $(n+1)!/2$ copies of the
Stasheff polytope $K_{n+1}$, we can use it to describe the quotient.  The
interior of $K_{n+1}$ corresponds to the classical elliptic coordinates
\eqref{eq:eigenvalues} on the sphere $S^n$.  The $n+1$ distinct real
parameters $(\Lambda_1,\dots,\Lambda_{n+1})\in\Sigma_{n+1}$ they depend on are
the eigenvalues of the symmetric tensor $\hat L$ on $\mathbb R^{n+1}$ which
restricts to the corresponding special conformal Killing tensor $L$ on $S^n$.
Shifting or scaling them only reparametrises the corresponding coordinates.
Hence the actual parameter space is the quotient $\Sigma_{n+1}/\mathrm{Aff}$,
which is nothing else but the open moduli space $M_{0,n+2}(\mathbb R)$. Thus
we have the following important

\begin{cor}
	\label{cor:degenerations}
	All orthogonal separation coordinates on $S^n$ belong to the closure of
	the Neumann family of elliptic coordinates.  The possible degenerations of
	the Neumann family correspond to the faces of the Stasheff polytope
	$K_{n+1}$.
\end{cor}

The first part is probably not surprising for the experts (see the similar
claim in the complex case in \cite{KMR}), but we have not seen it explicitly
stated and proved in the literature.  In Section~\ref{sec:operad} we will show
that rather than by actually performing the limiting process explicitly (as
in \cite{KMR}), the limiting cases can be better understood by composing
generic separation coordinates (that is elliptic coordinates) of lower
dimensions under the operad composition.  The same holds true for the
corresponding St\"ackel systems.

Because we have $(n+1)!/2$ Stasheff polytopes $K_{n+1}$ tiling $\bar
M_{0,n+2}(\mathbb R)$, the quotient $Y_n=\bar M_{0,n+2}(\mathbb R)/S_{n+1}$ is
actually only ``a half'' of $K_{n+1}$ with some identification between the
faces.  In the interior of the polytope the identification is given by the
action of $\mathbb Z_2 \subset D_{n+2}$, corresponding to a reflection in the
dihedral group, realised as an isometry of $K_{n+1}$ (see above). If we are
using the blow-up description of $\bar M_{0,n+2}(\mathbb R)$ \cite{Kap, Devadoss, DJS}, then it
corresponds to the longest element $(1,2,\dots, n, n+1) \mapsto (n+1, n,
\dots, 2, 1)$ in the symmetric group $S_{n+1}$, mapping the $A_n$ Weyl chamber
into the opposite one.  For $K_4$ this is just a reflection symmetry of the
pentagon as indicated in Figure~\ref{fig:coordinates}.

The identification of the faces is more sophisticated. Probably the best way
to describe them is using the ``twisting along the diagonal'' operation for
the dissected $(n+1)$-gon introduced in \cite{Devadoss, DR}.  On planar trees
this corresponds to reversing the ordering of a node's children.  Trees which
are equivalent under this operation are called ``dyslexic''.

As a matter of illustration, let us consider the least non-trivial example $n=2$, depicted in
Figure~\ref{fig:S2}.
\begin{figure}[h]
	\begin{center}
		\includegraphics[width=.75\textwidth]{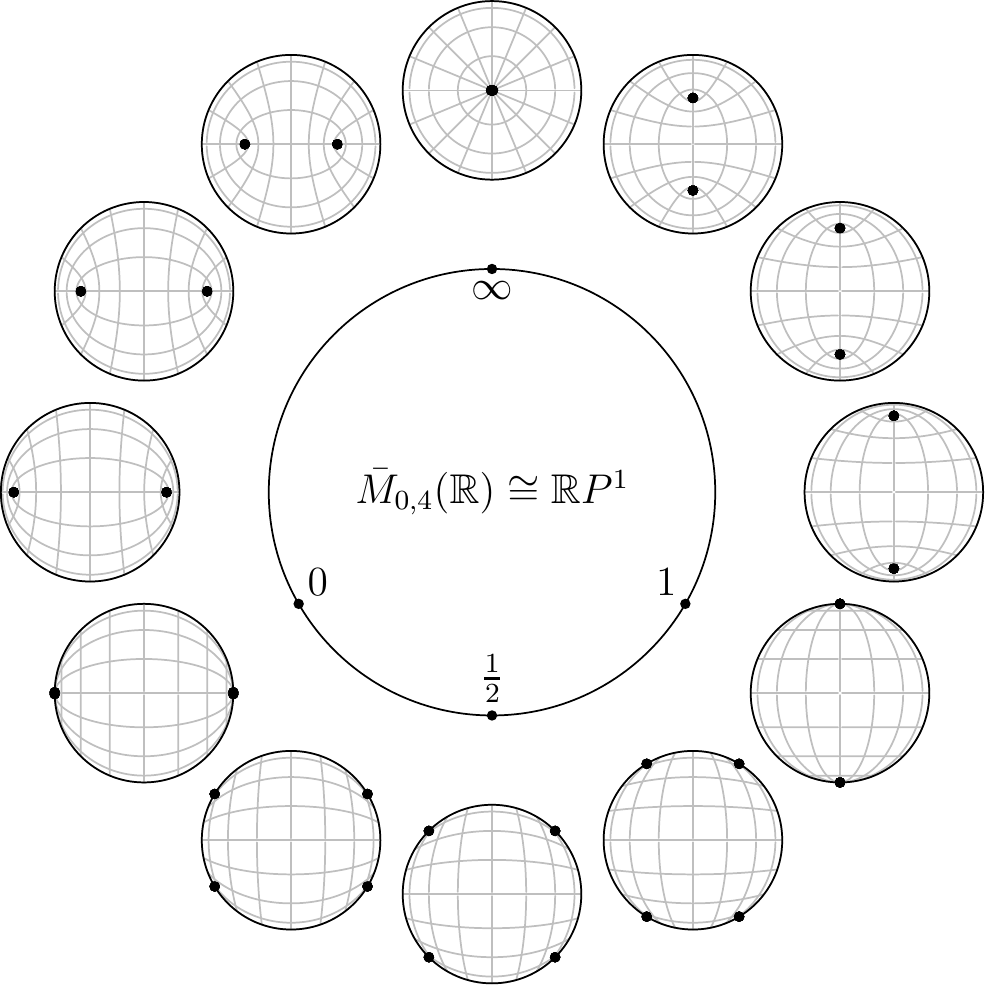}
		\caption{%
			Orthogonal separation coordinates on $S^2$ parametrised by $\bar
			M_{0,4}(\mathbb R)\cong\mathbb RP^1$.
		}
		\label{fig:S2}
	\end{center}
\end{figure}
The parameter space $\bar M_{0,4}(\mathbb R)\cong\mathbb RP^1$ is just a
circle, parametrised by the affine-invariant cross ratio
$\tau=\frac{\Lambda_2-\Lambda_1}{\Lambda_3-\Lambda_1}$.  It is tessellated by
three copies of the Stasheff polytope $K_3$, the tiles being the arcs between
the points $\tau=0,1,\infty$.  We can identify $K_3$ with the arc $[0,1]$, for
which $\Lambda_1<\Lambda_2<\Lambda_3$.  The symmetry group $S_3$ acts by
permuting $\Lambda_1$, $\Lambda_2$ and $\Lambda_3$ and hence the points $0$,
$1$ and $\infty$.  The quotient $Y_2=\bar M_{0,4}(\mathbb R)/S_3$ can be
identified with the arc $[0,\frac12]$.  This is ``a half'' of $K_3$, where
$\tau=0$ corresponds to spherical coordinates and $\tau=1/2$ to the
``lemniscatic'' case, which is just a particular case of elliptic coordinates
when $\Lambda_2=(\Lambda_1+\Lambda_3)/2$.  In this example there is no
identification in place, because the dimension is too low.

\section{Applications}

\subsection{Enumerating separation coordinates}

For a Stasheff polytope, the number of non-equivalent faces of a given
dimension can be given by the following Devadoss-Read formula \cite{DR}.  Let
$A(x,y)=\sum a_{mn}x^my^n$ be the formal series solution of the functional
equation
\begin{equation}\label{dr}
A(x,y)=y+\frac{1}{2}\left(\frac{A(x,y)^2}{1-A(x,y)}+\frac{(1+A(x,y))A(x^2,y^2)}{1-A(x^2,y^2)}\right).
\end{equation}
There is no closed formula for $A(x,y)$, but one uses Equation \eqref{dr} to find the coefficients $a_{mn}$ recursively.
The claim is that the coefficient $a_{mn}$ is the number of non-equivalent faces of $K_{n}$ of codimension $m-1$.
Devadoss and Read proved this using a combinatorial technique going back to P\'olya \cite{DR}.

Using Table~2 from \cite{DR}, we get the number of non-equivalent canonical
forms of separation coordinates on $S^n$ for $n\le 10$, as listed in
Table~\ref{tab:numbers}.
\begin{table}[h]
	\caption{%
		Number of canonical forms for separation coordinates on $S^n$, ordered
		by increasing number of independent continuous parameters.
	}
	\begin{tabular}{c|rrrrrrrrrr|r}
		\toprule
		& 0 & 1 & 2 & 3 & 4 & 5 & 6 & 7 & 8 & 9 & total \\
		\midrule
		$S^2$ & 1 & 1 & & & & & & & & & 2 \\
		$S^3$ & 2 & 3 & 1 & & & & & & & & 6 \\
		$S^4$ & 3 & 8 & 5 & 1 & & & & & & & 17 \\
		$S^5$ & 6 & 20 & 22 & 8 & 1 & & & & & & 57 \\
		$S^6$ & 11 & 49 & 73 & 46 & 11 & 1 & & & & & 191 \\
		$S^7$ & 23 & 119 & 233 & 206 & 87 & 15 & 1 & & & & 684 \\
		$S^8$ & 46 & 288 & 689 & 807 & 485 & 147 & 19 & 1 & & & 2482 \\
		$S^9$ & 98 & 696 & 1988 & 2891 & 2320 & 1021 & 236 & 24 & 1 & & 9275 \\
		$S^{10}$ & 207 & 1681 & 5561 & 9737 & 9800 & 5795 & 1960 & 356 & 29 & 1 & 35127 \\
		\bottomrule
	\end{tabular}
	\label{tab:numbers}
\end{table}
The 1's on the diagonal correspond to elliptic coordinates, the numbers in the
first column to polyspherical coordinates.  Note that the sequence
$1,2,3,6,11,23,\ldots$ is the sequence of Wedderburn-Etherington numbers,
i.\,e.\ the number of non-planar binary rooted trees with $n+1$ leaves
\cite{A001190}.  This reflects the fact that we can parametrise polyspherical
coordinates by planar binary rooted trees and that the notions \emph{dyslexic}
and \emph{non-planar} coincide for binary trees.

In the first row we have $1$ and $1$, corresponding to spherical and elliptic
coordinates on $S^2$ respectively, as discussed in
Section~\ref{sec:correspondence}.  The numbers in the second row -- $2$, $3$,
$1$ -- are in perfect agreement with the results of Eisenhart \cite{Eis},
Olevski \cite{Olev} and Kalnins \& Miller \cite{KM86}.  They correspond to
spherical and cylindrical coordinates (2), two types of Lam\'e rotational
coordinates plus Lam\'e subgroup reduction (3) and elliptic coordinates (1)
respectively.  Their identification with the faces of the Stasheff polytope
$K_4$ is indicated in Figure~\ref{fig:coordinates}, in comparison to the
different labellings shown in Figure~\ref{fig:combinatorics}.  Polyspherical
coordinates comprise, for example, the usual spherical coordinates plus
cylindrical coordinates.
\begin{figure}[h]
	\begin{center}
		\includegraphics[width=.618\textwidth]{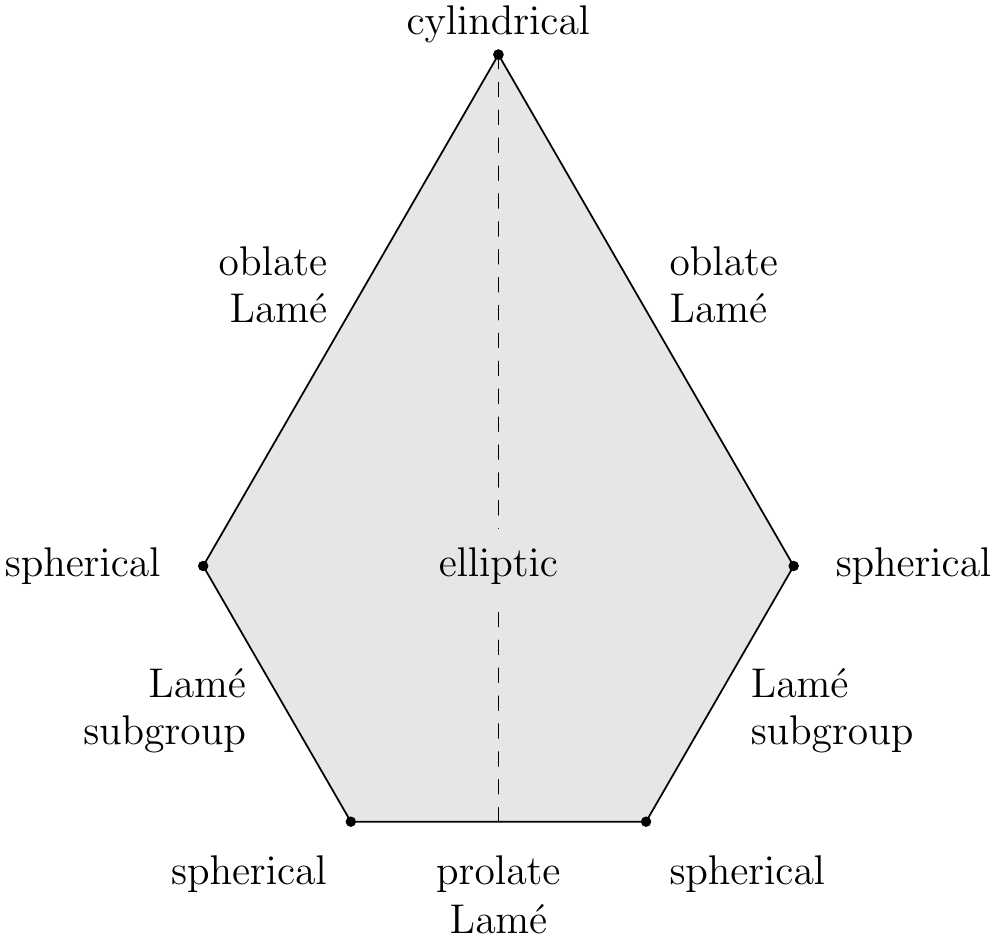}
		\caption{%
			The Stasheff polytope $K_4$, labelled by separation coordinates on
			$S^3$.
		}
		\label{fig:coordinates}
	\end{center}
\end{figure}
Note that ``the half'' of $K_4$, obtained as the quotient under the reflection
symmetry indicated in Figure~\ref{fig:coordinates}, is a quadrilateral and
that for the quotient space $Y_3=\bar M_{0,5}(\mathbb R)/S_4$ we have to identify
its two adjacent vertices that are labelled by spherical coordinates and
joined by Lam\'e subgroup reduction.

In accordance with our description, the total number of canonical forms for
separation coordinates on $S^n$, indicated in the last column in
Table~\ref{tab:numbers}, is the number of dyslexic planar rooted trees with
$n+1$ leaves \cite{A032132}.

\subsection{Constructing separation coordinates via the mosaic operad}
\label{sec:operad}

The correspondence in Theorem~\ref{thm:iso} transfers the natural operad
structure on $\bar M_{0,n+2}(\mathbb R)$ to orthogonal separation coordinates on $S^n$
and thereby yields a uniform construction procedure for separation coordinates
on spheres. Implicitly this structure is already present in \cite{KM86} (see, in particular, formula (3.14) therein).
Let us now make this operad structure explicit.  We begin with
the operad structure on $\mathcal O(n):=\mathbb R^n$ given by
\[
	\begin{array}{rrcl}
		\circ:&\mathbb R^k\times\mathbb R^{n_1}\times\ldots\times\mathbb R^{n_k}&\longrightarrow&\mathbb R^{n_1+\cdots+n_k}\\
		&(\boldsymbol y,\boldsymbol x_1,\ldots,\boldsymbol x_k)&\mapsto&\boldsymbol y\circ(\boldsymbol x_1,\ldots,\boldsymbol x_k):=(y_1\boldsymbol x_1,\ldots,y_k\boldsymbol x_k)
	\end{array}
\]
together with the permutation action of $S_n$ on $\mathbb R^n$.  This operad
structure descends to an operad structure on $\mathcal
O(n):=S^{n-1}\subset\mathbb R^n$, since $(y_1\boldsymbol
x_1)^2+\cdots(y_k\boldsymbol x_k)^2=y_1^2+\cdots y_k^2=1$ for $\boldsymbol
y\in S^{k-1}$ and $\boldsymbol x_\alpha\in S^{n_\alpha-1}$.  Note that the
composition map
\[
	\circ:S^{k-1}\times S^{n_1-1}\times\cdots\times S^{n_k-1}\longrightarrow S^{n_1+\cdots+n_k-1}
\]
describes the $k$-fold join $S^{n_1-1}\star\cdots\star S^{n_k-1}\cong
S^{n_1+\cdots+n_k-1}$ of the spheres $S^{n_1-1}$, \ldots, $S^{n_k-1}$.  This
operad structure on spheres induces an operad structure on (local) coordinates
on spheres, since coordinates $\boldsymbol x_0=\boldsymbol
x_0(\varphi_{01},\ldots,\varphi_{0,k-1})$ on $S^{k-1}\subset\mathbb R^k$ and
$\boldsymbol x_\alpha=\boldsymbol
x_\alpha(\varphi_{\alpha1},\ldots,\varphi_{\alpha,n_\alpha-1})$ on
$S^{n_\alpha-1}\subset\mathbb R^{n_\alpha}$ for $\alpha=1,\ldots,k$ determine
coordinates
\[
	\boldsymbol x=\boldsymbol x(\varphi_{01},\ldots,\varphi_{0,k-1},\varphi_{11},\ldots,\varphi_{1,n_1-1},\ldots,\varphi_{k1},\ldots,\varphi_{k,n_k-1})
\]
on $S^{n_1+\cdots+n_k-1}$, given by setting
\begin{equation}
	\label{eq:composition}
	\boldsymbol x:=\boldsymbol x_0\circ(\boldsymbol x_1,\ldots,\boldsymbol x_k).
\end{equation}
The interior of a Stasheff polytope corresponds to elliptic coordinates and
its faces are products of Stasheff polytopes.  Therefore we can construct all
orthogonal separation coordinates on spheres (modulo isometries) from elliptic
coordinates by composing them in a recursive manner via the operad composition
\eqref{eq:composition}.  Just start with trivial
coordinates $\boldsymbol x(\varnothing)=1$ on a certain number of zero
dimensional spheres $S^0$ and take elliptic coordinates for $\boldsymbol x_0$
in each step.  The different choices one has when iterating this composition
are given by the trees labelling the corresponding separation coordinates.
That is, the rooted trees describe the hierarchy of iterated decompositions of
a sphere as joins of lower dimensional spheres.  This parallels the
decomposition of the faces of a Stasheff polytope into products of lower
dimensional Stasheff polytopes.

Note that the construction \eqref{eq:polyspherical} of Vilenkin's
polyspherical coordinates on $S^{n-1}$ corresponds to the special case $k=2$ of
the above construction, starting from trivial coordinates $\boldsymbol
x(\varnothing)=1$ on $n$ copies of $S^0$ and using the (elliptic) coordinates
$\boldsymbol x_0(\varphi)=(\cos\varphi,\sin\varphi)$ on $S^{k-1}=S^1$ in each
step.

Moreover, this operad structure on separation coordinates also explains
Kalnins \& Miller's graphical procedure \cite{KM86}.  Namely, adding in an
``irreducible block'' a leaf to each box which is not joined to another block
and replacing each irreducible block by a node, the graphs in \cite{KM86}
become the trees arising from the operad structure.

\subsection{Constructing St\"ackel systems via the mosaic operad}

We now explain how this operad structure manifests itself on the level of
St\"ackel systems.  To this end, let $I_1\cup\cdots\cup I_k=I$ be a partition
of $I=\{1,\ldots,n\}$ with $\lvert I_\alpha\rvert=:n_\alpha$ and set
$I_0:=\{1,\ldots,k\}$.  We denote by $\mathfrak d_n^1$ the space of Killing
tensors on $S^{n-1}$ with diagonal algebraic curvature tensor and define the
injections
\begin{align*}
	&\iota_0:\quad\mathfrak d_k^1\;\;\,\hookrightarrow\mathfrak d_{n_1+\cdots+n_k}^1&
	\iota_0(K_{\alpha\beta})
	&:=
	\mspace{-16mu}\sum\limits_{a\in I_\alpha,b\in I_\beta}\mspace{-16mu}K_{ab}&
	\alpha,\beta&\in I_0\\
	&\iota_\alpha:\quad\mathfrak d_{n_\alpha}^1\hookrightarrow\mathfrak d_{n_1+\cdots+n_k}^1&
	\iota_\alpha(K_{ij})&:=K_{ij}&
	i,j&\in I_\alpha.
\end{align*}
\begin{prop}
	Let $\Sigma_0$ be a St\"ackel system on $S^{k-1}$ and $\Sigma_\alpha$ be St\"ackel
	systems on $S^{n_\alpha-1}$ for $\alpha=1,\ldots,k$, all consisting of
	Killing tensors with diagonal algebraic curvature tensor.  Then
	\begin{equation}
		\label{eq:Staeckel-operad}
		\Sigma_0\circ(\Sigma_1,\ldots,\Sigma_k):=\iota_0(\Sigma_0)\oplus\iota_1(\Sigma_1)\oplus\cdots\oplus\iota_k(\Sigma_k)
	\end{equation}
	is a St\"ackel system on $S^{n_1+\cdots+n_k-1}$.  Moreover, this operation
	together with the $S_n$-action \eqref{eq:groupaction} defines an
	operad structure on those St\"ackel systems on $S^{n-1}$ that consist of
	Killing tensors with diagonal algebraic curvature tensor.
\end{prop}
\begin{proof}
	First observe that the sum on the right hand side of
	\eqref{eq:Staeckel-operad} is indeed a direct sum.  Hence its dimension is
	$n_1+\cdots+n_k-1$.  By Definition~\ref{def:Staeckel} and
	Remark~\ref{rem:commutator-equivalence}, we have to show that all Killing
	tensors in this subspace commute.  That is, a Killing tensor from
	$\iota_p(\Sigma_p)$ and another one from $\iota_q(\Sigma_q)$ commute for all
	$p,q=0,1,\ldots,k$.  For $p,q\not=0$ this is evident.  For $p=q=0$ one
	readily checks that the inclusion $\iota_0$ preserves the relations
	\eqref{eq:commutator:relations} and hence maps commuting Killing tensors
	to commuting Killing tensors.  In the remaining case $p\not=0=q$ the
	commutator
	\[
		[\iota_p(K_{ij}),\iota_0(K_{\alpha\beta})]
		=[K_{ij},\mspace{-16mu}\sum\limits_{a\in I_\alpha,b\in I_\beta}\mspace{-16mu}K_{ab}]
		\qquad
		i,j\in I_p,
		\quad
		\alpha\not=\beta\in I_0
	\]
	is zero unless $p=\alpha$ or $p=\beta$.  But if $p=\alpha$, the sum over
	$a\in I_\alpha$ only contributes non-zero terms for $a=i$ and $a=j$.  So
	the above commutator reduces to
	\[
		\sum_{b\in I_\alpha}[K_{ij},K_{ib}+K_{jb}]=0
	\]
	due to the relations \eqref{eq:commutator:relations}, and similarly for
	$p=\beta$.  This proves that \eqref{eq:Staeckel-operad} is a St\"ackel
	system.

	To check that this composition defines an operad is straightforward.  The
	identity element is the empty St\"ackel system on $S^0$ and equivariance
	is obvious.  Associativity can be shown by taking subdivisions
	$I_{\alpha}=I_{\alpha1}\cup\cdots\cup I_{\alpha k_\alpha}$ of $I_\alpha$
	for all $\alpha\in I_0$ and considering the corresponding inclusions for
	Killing tensors.  The details will be left to the reader.
\end{proof}

To give an example, let us construct the St\"ackel system for standard
spherical coordinates on $S^{n-1}$ by choosing $k=2$ with the St\"ackel system
$\Sigma_0$ on $S^{k-1}=S^1$ spanned by $K_{12}$, starting from empty St\"ackel
systems on $n$ copies of $S^0$ and taking $n_2=1$ in each step.  This yields
the St\"ackel system spanned by
\[
	K_{12},\quad K_{13}+K_{23},\quad K_{14}+K_{24}+K_{34},\quad\dots
\]
and shows that the Jucys-Murphy subalgebras \eqref{eq:Jucys-Murphy} in the
Kohno-Drinfeld Lie algebra correspond to standard spherical coordinates.

\section{Outlook}

We have shown that the theory of Deligne-Mumford-Knudsen moduli spaces and
Stasheff polytopes provides the right framework for the classification and
construction of all orthogonal separation coordinates on spheres.  In particular, we
elucidated the natural algebro-geometric structure of the parameter space
classifying isometry classes of separation coordinates, which for a long time
had only been known as a mere set, and gave a precise description of its
topology.

It would be very interesting to see if the same approach will work in a more
general situation. In particular, one can use the algebraic approach of
\cite{Schoebel} to study the orthogonal separation coordinates for all
(pseudo-)Riemannian constant curvature manifolds, such as hyperbolic space
$\mathbb H^n$.  The question is whether the corresponding moduli spaces of
separation coordinates are related to any known algebro-geometric moduli
spaces or families of polyhedra.  This will be subject for future research.

\section{Acknowledgements}

Both of us would like to thank the Hausdorff Research Institute for
Mathematics in Bonn for its hospitality in March 2012, when this work was
essentially done.  We are also grateful to Prof.\ S.~Devadoss, who attracted
our attention to the paper \cite{DR}, and to Prof.\ J.~Stasheff for his
very useful comments.

\end{document}